\documentclass[11pt, a4paper]{article}
\usepackage{amsmath, amsthm, amssymb, url}
\usepackage[margin=1in]{geometry}
\usepackage[english]{babel}
\usepackage{multirow}
\usepackage{authblk}
\newcommand{\End}{\mathrm{End}}

\newcommand{\id}{\mathrm{id}}

\newcommand{\Res}{\mathrm{Res}}
\newcommand{\CA}{\mathrm{CA}}
\newcommand{\ICA}{\mathrm{ICA}}
\newcommand{\GCA}{\mathrm{GCA}}
\newcommand{\IGCA}{\mathrm{IGCA}}

\newcommand{\Hom}{\mathrm{Hom}}
\newcommand{\Fix}{\mathrm{Fix}}
\newcommand{\Aut}{\mathrm{Aut}}

\newcommand{\op}{\mathrm{op}}

\theoremstyle{plain}

\newtheorem{corollary}{Corollary}
\newtheorem{lemma}{Lemma}
\newtheorem{proposition}{Proposition}
\newtheorem{theorem}{Theorem}

\theoremstyle{definition}
\newtheorem{definition}{Definition}

\newtheorem{example}{Example}

\newtheorem{question}{Question}

\begin{document}

\title{A generalization of cellular automata over groups}
\author[1]{A. Castillo-Ramirez\footnote{Email: alonso.castillor@academicos.udg.mx}}
\author[2]{M. Sanchez-Alvarez}
\author[1]{A. Vazquez-Aceves}
\author[1]{A. Zaldivar-Corichi} 
\affil[1]{Centro Universitario de Ciencias Exactas e Ingenier\'ias, Universidad de Guadalajara, M\'exico.}
\affil[2]{Centro Universitario de Los Valles, Universidad de Guadalajara, M\'exico.}

\maketitle

\begin{abstract}
Let $G$ be a group and let $A$ be a finite set with at least two elements. A cellular automaton (CA) over $A^G$ is a function $\tau : A^G \to A^G$ defined via a finite memory set $S \subseteq G$ and a local function $\mu :A^S \to A$. The goal of this paper is to introduce the definition of a generalized cellular automaton (GCA) $\tau : A^G \to A^H$, where $H$ is another arbitrary group, via a group homomorphism $\phi : H \to G$. Our definition preserves the essence of CA, as we prove analogous versions of three key results in the theory of CA: a generalized Curtis-Hedlund Theorem for GCA, a Theorem of Composition for GCA, and a Theorem of Invertibility for GCA. When $G=H$, we prove that the group of invertible GCA over $A^G$ is isomorphic to a semidirect product of $\Aut(G)^{\op}$ and the group of invertible CA. Finally, we apply our results to study automorphisms of the monoid $\CA(G;A)$ consisting of all CA over $A^G$. In particular, we show that every $\phi \in \Aut(G)$ defines an automorphism of $\CA(G;A)$ via conjugation by the invertible GCA defined by $\phi$, and that, when $G$ is abelian, $\Aut(G)$ is embedded in the outer automorphism group of $\CA(G;A)$. 
 \\

\end{abstract}
\section{Introduction}\label{intro}

Cellular automata are functions between prodiscrete topological spaces with the key features of being determined by a finite memory set and a fixed local function. They have become fundamental objects in several areas of mathematics, such as symbolic dynamics, complexity theory, and complex systems modeling, and its theory has flourished due to its diverse connections with group theory, topology, and dynamics (e.g. see the influential book \cite{CSC10} and references therein). 

If $G$ is a group and $A$ is a finite set, known as an alphabet, the \emph{configuration space} over $G$ and $A$, denoted by $A^G$, is the set of all functions $x : G \to A$. We endow $A^G$ with the \emph{prodiscrete topology}, i.e. the product topology of the discrete topology of $A$, and with the \emph{shift action} of $G$ on $A^G$ given by 
\[ g \cdot x(h) := x(g^{-1}h), \quad \forall x \in A^G, g,h \in G. \]
A \emph{cellular automaton} over $A^G$ is a function $\tau : A^G \to A^G$ such that there is a finite subset $S \subseteq G$, called a \emph{memory set} of $\tau$, and a \emph{local function} $\mu : A^S \to A$ satisfying
\begin{equation}\label{defCA}
\tau(x)(g) = \mu (( g^{-1} \cdot x) \vert_{S}),  \quad \forall x \in A^G, g \in G. 
\end{equation}

The goal of this paper is to generalize the above definition in order to allow the existence of cellular automata from $A^G$ to $A^H$, with $G$ and $H$ two arbitrary groups. We believe this is a meaningful question, not only because it enriches the theory of cellular automata, but also because it extends their applicability to problems that consider the interaction between configuration spaces over different groups, or over a twisted version of the same group. 

For any group homomorphism $\phi : H \to G$, a \emph{$\phi$-cellular automaton} is a function $\tau : A^G \to A^H$ such that there exists a memory set $S \subseteq G$ and a local function $\mu :A^S \to A$ satisfying 
\begin{equation}\label{defGCA}
\tau(x)(h) = \mu (( \phi(h^{-1}) \cdot x) \vert_{S}),  \quad \forall x \in A^G, h \in H. 
\end{equation}

We remark that our generalized definition goes in a different sense than the so-called \emph{sliding block codes} \cite[Sec. 1.5]{LM95}. These latter functions may be defined between two arbitrary \emph{subshifts} $X$ and $Y$ of $A^G$ (i.e. closed and $G$-invariant subsets), but the group acting in the domain and codomain must be the same.  

Our generalized definition of celular automata may be further generalized, for example, by considering a  different alphabet in the codomain of the function, by considering different subshifts in the domain and codomain, or by dropping the assumption that the alphabet $A$ is finite. In these situations, it is straightforward to obtain analogous results to some of the ones obtained in this paper; for example, if $A$ is not necessarily finite, we may use the prodiscrete uniform structure of $A^G$, as in \cite[Sec. 1.9]{CSC10}, to obtain a generalization of Curtis-Hedlund theorem. However, we believe that these additional generalizations may obscure the essence of our study, so we have decided not to include them in our definition. 

The structure of this paper is as follows. In Section 2, prove that every $\phi$-cellular automaton $\tau : A^G \to A^H$ is \emph{$\phi$-equivariant} in the sense that 
\[ \tau( \phi(h) \cdot x) = h \cdot \tau(x), \quad \forall h \in H, x \in A^G. \]
When $G=H$, $\id_G$-equivariance coincides with the usual property of $G$-equivariance that is satisfied by classical cellular automata. The notion of $\phi$-equivariance much resembles the concept of \emph{homomorphism of group actions} \cite[Def. 2.1]{GLM}, except that the direction of the group homomorphism $\phi : H \to G$ is reversed. As a next step, we prove three results that are analogous to three key theorems in the theory of classical cellular automata: 
\begin{enumerate}
\item \emph{Generalized Curtis-Hedlund Theorem:} A function $\tau : A^G \to A^H$ is a $\phi$-cellular automaton if and only if it is $\phi$-equivariant and continuous. 

\item \emph{Theorem of Composition:} The composition $\sigma \circ \tau$ of a $\psi$-cellular automaton $\sigma$ with memory set $S$ with a $\phi$-cellular automaton $\tau$ with memory set $T$ is a $(\phi \circ \psi)$-cellular automaton with memory set $\phi(S)T$.

\item \emph{Theorem of Invertibility:} A $\phi$-cellular automaton $\tau : A^G \to A^H$ is \emph{invertible} (in the sense that there exits a group homomorphism $\psi : G \to H$ and a $\psi$-cellular automaton $\sigma : A^H \to A^G$ such that $\tau \circ \sigma = \id_{A^H}$ and $\sigma \circ \tau = \id_{A^G}$) if and only if $\tau$ is bijective. 
\end{enumerate}

In Section 3, we focus on the monoid $\GCA(G;A)$, consisting of the set of all $\phi$-cellular automata $\tau : A^G \to A^G$ for some $\phi \in \End(G)$, equipped with composition of functions, and its group of invertible elements $\IGCA(G;A)$. We show that $\GCA(G;A)$ has a submonoid isomorphic to $\End(G)^{\op}$, and that every $\tau \in \GCA(G;A)$ may be expressed as the composition of an element in this isomorphic copy of  $\End(G)^{\op}$, and an element in the monoid $\CA(G;A)$ of classical cellular automata over $A^G$. Furthermore, we show that the group $\IGCA(G;A)$ is isomorphic to the semidirect product of $\Aut(G)^{\op}$ and the group of classical invertible cellular automata $\ICA(G;A)$.  

Finally, in Section 4, we study automorphisms of $\CA(G;A)$. Clearly, every element of $\ICA(G;A)$ induces by conjugation an \emph{inner automorphism} of $\CA(G;A)$, and this defines the normal subgroup $\text{Inn}(\CA(G;A))$ of $\Aut(\CA(G;A))$. It turns out that conjugation by an invertible $\phi$-cellular automaton also defines an automorphism of $\CA(G;A)$; such is the case, for example, of the so-called \emph{mirrored rule}, or \emph{reflection}, which is widely used in the study of elementary cellular automata over $A^{\mathbb{Z}}$ (e.g., see \cite[Ch. 20]{Wolfram}). We establish that, when $G$ is abelian, the outer automorphism group $\text{Out}(\CA(G;A)) := \Aut(\CA(G;A)) / \text{Inn}(\CA(G;A))$ contains a subgroup isomorphic to $\Aut(G)$, and we pose the following question: perhaps under some additional hypothesis, is it possible to show that $\text{Out}(\CA(G;A))$ is isomorphic to $\Aut(G)$, or can we always find an outer automorphism of $\CA(G;A)$ that is not induced by an automorphism of $G$?


\section{A generalized definition} 

We assume that the reader knows the fundamentals of group theory, topology and the theory of classical cellular automata over groups (see \cite[Ch. 1]{CSC10}). 

For the rest of the paper, let $A$ be a finite set, and let $G$ and $H$ be groups. As the case $\vert A \vert = 1$ is trivial and degenerate, we shall assume that $\vert A \vert \geq 2$ and that $\{ 0,1 \} \subseteq A$. Denote by $\Hom(H,G)$ the set of all group homomorphisms from $H$ to $G$.

\begin{definition}
For any $\phi \in \Hom(H,G)$, a $\phi$-\emph{cellular automaton} from $A^G$ to $A^H$ is a function $\tau : A^G \to A^H$ such that there is a finite subset $S \subseteq G$, called a \emph{memory set} of $\tau$, and a \emph{local function} $\mu : A^S \to A$ satisfying
    \[ \tau(x)(h) = \mu (( \phi(h^{-1}) \cdot x) \vert_{S}),  \quad \forall x \in A^G, h \in H.  \]
\end{definition}

\begin{example}
Every cellular automaton from $A^G$ to $A^G$ is an $\id_G$-cellular automaton, where $\id_{G}$ is the identity function on $G$. However, note that we may define $\phi$-cellular automata from $A^G$ to $A^G$, where $\phi$ is a nontrivial element of $\End(G) := \Hom(G,G)$.  
\end{example}

\begin{example}
The homomorphism that defines a $\phi$-cellular automaton may be not unique. For example, for any $S \subseteq G$, let $\mu : A^S \to A$ be a constant function. Then, every $\phi \in \Hom(H,G)$ defines the same $\phi$-cellular automaton with memory set $S$ and local function $\mu$.  
\end{example}

\begin{example}
Let $G = \mathbb{Z}$, $H = \mathbb{Z}^2$ and $S = \{ -1,0,1 \} \subseteq \mathbb{Z}$. Recall that a configuration $x \in A^{\mathbb{Z}}$ may be seen as a bi-infinite sequence $x = \dots x_{-1}, x_0, x_1, \dots$. 

Consider the homomorphism $\phi : \mathbb{Z}^2 \to \mathbb{Z}$ given by $\phi(a,b) = a+b$, for all $(a,b) \in \mathbb{Z}^2$. Then, for any function $\mu : A^S \to A$, the $\phi$-cellular automaton $\tau : A^{\mathbb{Z}} \to A^{\mathbb{Z}^2}$ with memory set $S$ and local function $\mu$ is given by
\[ \tau(x)(a,b) = \mu( x_{a+b- 1}, x_{a+b}, x_{a+b+1}).   \] 
for all $x \in A^\mathbb{Z}$ and $(a,b) \in \mathbb{Z}^2$. 
\end{example}

\begin{example}\label{star}
For every $\phi \in \Hom(H,G)$, define $\phi^{\star} : A^G \to A^H$ by 
\[ \phi^\star(x) = x \circ \phi, \quad \forall x \in A^G. \]
Observe that for any $h \in H$,
\[ \phi^\star(x)(h) = x \circ \phi(h) = (\phi(h^{-1}) \cdot x)( e_G ) , \]
where $e_G$ is the identity of $G$. Hence, $\phi^\star$ is a $\phi$-cellular automaton with memory set $S=\{ e_G \}$ and local function $\mu = \id_A$.  
\end{example}

If $K$ is a subgroup of $G$, recall that a configuration in $A^G$ is $K$-periodic if $k \cdot x = x$, for all $k \in K$. Denote the set of all $K$-periodic configurations in $A^G$ by $\Fix(K)$. 

\begin{example}\label{periodic}
Consider the cyclic groups $G = \mathbb{Z}_n$ and $H = \mathbb{Z}$. Let $\phi : \mathbb{Z} \to \mathbb{Z}_n$ be the homomorphism given by $\phi(k) = k \mod(n)$. Then, $\phi^{\star} : A^{\mathbb{Z}_n} \to A^{\mathbb{Z}}$ sends $n$-tuples of $A$ to $n \mathbb{Z}$-periodic configurations. For example, if $n=3$, then 
\[ \phi^\star(1,0,0) = \dots 1, 0, 0, 1, 0, 0 , 1,0,0 \dots. \]   
\end{example}

Classical cellular automata have the important property of being $G$-equivariant, i.e. $\tau(g \cdot x) = g \cdot \tau(x)$ for every $g \in G$, $x \in A^G$. However, $\phi$-celllular automata satisfy a more general property.

\begin{definition}\label{equiv}
Let $\phi \in \Hom(H,G)$. A function $\tau : A^G \to A^H$ is called \emph{$\phi$-equivariant} if
    \[ h\cdot \tau(x) = \tau(\phi(h)\cdot x),  \quad \forall x \in A, h \in H.  \]
 \end{definition}
 
The above definition much resembles the following concept: if $X$ is a $G$-set and $Y$ is a $H$-set, a \emph{homomorphism of group actions} is a function $f : X \to Y$ together with a group homomorphism $\alpha : G \to H$ such that $f(g \cdot x) = \alpha(g) \cdot f(x)$, for all $g \in G$, $x \in X$ (see \cite[Def. 2.1]{GLM}). Note that in Definition \ref{equiv}, the direction of the group homomorphism is reversed. However, these two concepts coincide in the case of isomorphisms: a $G$-set $X$ is isomorphic to an $H$-set $Y$ if and only if there exists an isomorphism $\phi : H \to G$ and a $\phi$-equivariant bijection $\tau : X \to Y$ (c.f. \cite[p. 17]{Dixon}). 

\begin{lemma}\label{le-equivariance}
Every $\phi$-cellular automaton is $\phi$-equivariant.
\end{lemma}
\begin{proof}
    Let $\tau : A^G \to A^H$ be a $\phi$-cellular automaton, $S\subseteq G$ a memory set of $\tau$, and $\mu : A^S \to A$ its local function. For all $x\in A^G$ and $k,h\in H$, we have
    \begin{eqnarray*}
        k\cdot\tau(x)(h) &=& \tau(x)(k^{-1}h)\\
        &=& \mu (( \phi(h^{-1})\phi(k) \cdot x) \vert_{S}) \\
        & =&  \mu (( \phi(h^{-1})\cdot(\phi(k) \cdot x)) \vert_{S}) \\
        &=& \tau(\phi(k)\cdot x)(h)
    \end{eqnarray*}
\end{proof}

Example \ref{periodic} may be generalized to show that $\phi$-cellular automata always map configurations to $\ker(\phi)$-periodic configurations (c.f. \cite[Prop. 1.3.7]{CSC10}).

\begin{proposition}\label{fix}
Let $\phi \in \Hom(H,G)$ and let $N := \ker(\phi)$. 
\begin{enumerate}
\item If $\tau : A^G \to A^H$ is a $\phi$-equivariant function, then $\text{Im}(\tau) \subseteq \Fix(N)$. 
\item If $\phi^\star : A^G \to A^H$ is as defined in Example \ref{star}, then $\text{Im}(\phi^\star) = \Fix(N)$.
\end{enumerate}
\end{proposition}
\begin{proof}
\begin{enumerate}
\item For any $n \in N$ and $x \in A^G$, we have by $\phi$-equivariance 
\[ n \cdot \tau(x) = \tau( \phi(n) \cdot x) = \tau(e \cdot x) = \tau(x). \]
This shows that $\tau(x) \in \Fix(N)$, for all $x \in A^G$. 

\item The first part of this proposition shows that $\text{Im}(\phi^\star) \subseteq \Fix(N)$. Now, let $x \in \Fix(N)$. Consider the quotient group $H /N$, and observe that for any $nh \in Nh$ we have 
\[ x(nh) = (n^{-1} \cdot x) (h) = x(h). \]
This means that $x$ is constant on each coset $Nh \in H/N$. By the First Isomorphism Theorem, there is an isomorphism $\Psi : \text{Im}(\phi) \to H/N$ given by $\Psi(\phi(h)) = Nh$ for all $h \in H$. Define $z \in A^G$ as follows: if $g \in \text{Im}(\phi)$, then $z(g) :=x(h)$, where $h \in H$ is any representative of the coset $\Psi(g)$, and if $g \in G - \text{Im}(\phi)$, then $z(g) = 0$. We claim that $z$ is a preimage of $x$ under $\phi^\star$ because for any $h \in H$, 
\[ \phi^{\star}(z) (h) = z( \phi(h) ) = x(h).    \]  
This shows that $x \in \text{Im}(\phi^\star)$, and the result follows. 
\end{enumerate}
\end{proof}

For each $h \in H$, let $\pi_h : A^H \to A$ be the projection map defined by $\pi_h(x) := x(h)$, for all $x \in A^H$. Projections are always continuous functions and they satisfy that a map $\tau : A^G \to A^H$ is continuous if and only if $\pi_h \circ \tau : A^G \to A$ is continuous for all $h \in H$. 

\begin{lemma}\label{le-continuity}
	Every $\phi$-cellular automaton is continuous.
\end{lemma}
\begin{proof}
	Let $\tau : A^G \to A^H$ be a $\phi$-cellular automaton with memory set $S \subseteq G$ and local function $\mu : A^S \to A$. Note that for all $x\in A^G$ and $h\in H$, we have
	\[ (\pi_h\circ\tau)(x)=\tau(x)(h)=\mu (( \phi(h^{-1}) \cdot x)\vert_{S})=\mu\circ\Res_S\circ\varphi_{\phi(h^{-1})}(x), \]
where $\Res_S:A^G\to A^S$ is the restriction function, and $\varphi_g:A^G\to A^G$ is defined by $\varphi_g(x) = g \cdot x$, for all $g \in G$, $x \in A^G$. It is easy to check that $\mu$, $\Res_S$, and $\varphi_g$ are all continuous, so $\pi_h\circ\tau$ must be continuous for every $h \in H$. This implies that $\tau$ is continuous. 
\end{proof}

Recall from \cite[Sec. 1.2]{CSC10} that a neighborhood base of $x \in A^G$ is given by the sets
\[  V(x,S) := \{ y \in A^G : x \vert_S = y \vert_S \}, \]
where $S$ runs among all finite subsets of $G$. The next one is a technical lemma that is useful to prove the main results of this section. 

\begin{lemma}\label{le-constant}
Let $S$ be a finite subset of $G$ and $\phi \in \Hom(H,G)$. A $\phi$-equivariant function $\tau : A^G \to A^H$ is a $\phi$-cellular automaton with memory set $S$ if and only if the function $\pi_{e_G} \circ \tau : A^G \to A$ is constant on $V(x,S)$, for every $x \in A^G$. 
\end{lemma}
\begin{proof}
Suppose that $\tau$ is a $\phi$-cellular automaton with memory set $S$. Let $\mu : A^S \to A$ be the local function of $\tau$. Then, for every $x \in A^G$ and $y \in V(x,S)$ we have
\[ \pi_{e_G} \circ \tau (x)   = \tau(x)(e_G) = \mu( x \vert_S) = \mu( y \vert_S) = \tau(y)(e_G) = \pi_{e_G} \circ \tau (y). \]   
Therefore, $\pi_{e_G} \circ \tau$ is constant on $V(x,S)$, for every $x \in A^G$

On the other hand, suppose that $\pi_{e_G} \circ \tau$ is constant on $V(x,S)$, for every $x \in A^G$. Define a function $\mu : A^S \to A$ as follows: for every $z \in A^S$, let $\mu(z) := \tau(z^\prime)(e_G)$, where $z^\prime \in A^G$ is any configuration such that $z^\prime \vert_{S} = z$. Now, for every $x \in A^G$ and $h \in H$, we have by $\phi$-equivariance,
\[ \tau(x)(h) = h^{-1} \cdot \tau(x)(e_G)  = \tau( \phi(h^{-1}) \cdot x) (e_G) = \mu( ( \phi(h^{-1}) \cdot x) \vert_S). \]     
This shows that $\tau$ is a $\phi$-cellular automaton with memory set $S$. 
\end{proof}

The following is a generalized version of Curtis-Hedlund Theorem \cite[Theorem 1.8.1]{CSC10}. 

\begin{theorem}\label{Curtis-Hedlund}
	Let $\phi \in \Hom(H,G)$. A function $\tau : A^G \to A^H$ is a $\phi$-cellular automaton if and only if $\tau$ is continuous and $\phi$-equivariant.
\end{theorem}
\begin{proof}
If $\tau$ is a $\phi$-cellular automaton, the result follows by Lemmas \ref{le-equivariance} and \ref{le-continuity}.

Suppose that $\tau : A^G \to A^H$ is continuous and $\phi$-equivariant. As in the proof of \cite[Theorem 1.8.1]{CSC10}, the compactness of $A^G$ implies that there exists a finite set $S \subseteq G$ such that $\pi_{e_G} \circ \tau$ is constant on $V(x,S)$, for all $x \in A^G$. By Lemma \ref{le-constant}, $\tau$ is a $\phi$-cellular automaton with memory set $S$. 
\end{proof}

The following is a generalized version of \cite[Proposition 1.4.9]{CSC10}

\begin{theorem}\label{composition}
Let $G$, $H$ and $K$ be groups, and consider $\phi \in \Hom(H,G)$ and $\psi \in \Hom(K,H)$. Let $\tau : A^G \to A^H$ be a $\phi$-cellular automaton with memory set $T \subseteq G$ and let $\sigma : A^H \to A^K$ be a $\psi$-cellular automaton with memory set $S \subseteq H$. Then, $\sigma\circ\tau:A^G\to A^K$ is a $(\phi\circ\psi)$-cellular automaton with memory set $\phi(S)T := \{ \phi(s)t : s \in S, t \in T \}$.
\end{theorem}
\begin{proof}
First note that for every $k \in K$ and $x \in A^G$,
\[ k \cdot ( \sigma\circ\tau (x)) = \sigma( \psi(k) \cdot \tau(x) ) = \sigma \circ \tau ( \phi(\psi(k)) \cdot x) .  \]
This shows that $\sigma \circ \tau$ is $(\phi \circ \psi)$-equivariant. We shall show that $\pi_{e_G} \circ \sigma \circ \tau$ is constant on $V(x, \phi(S)T)$, for all $x \in A^G$. Fix $x \in A^G$  and $y \in V(x, \phi(S)T)$. Observe that $\phi(S)T = \bigcup \limits_{s \in S} \phi(s)T$ and so 
\[ V(x, \phi(S)T) = V\left(x, \bigcup_{s \in S} \phi(s)T \right) = \bigcap_{s \in S} V(x,\phi(s)T).  \]
Hence $y \in V(x,\phi(s)T)$ for all $s \in S$, which implies that $\phi(s^{-1}) \cdot y \in V( \phi(s^{-1}) \cdot x, T)$. By Lemma \ref{le-constant} applied to $\tau$, we have
\[ \tau(\phi(s^{-1}) \cdot y)(e_G) = \tau(\phi(s^{-1}) \cdot x)(e_G).    \]
By the $\phi$-equivariance of $\tau$, 
\[ \tau(y)(s)  =s^{-1} \cdot \tau(y)(e_G)   =  \tau(\phi(s^{-1}) \cdot y)(e_G) = \tau(\phi(s^{-1}) \cdot x)(e_G) = s^{-1} \cdot \tau(x)(e_G) = \tau(x)(s) , \]
for every $s \in S$. This means that $\tau(y) \in V(\tau(x), S)$. By Lemma \ref{le-constant} applied to $\sigma$, we have $\sigma(\tau(y))(e_G) = \sigma(\tau(x))(e_G)$, as required. Therefore, $\sigma \circ \tau$ is a $(\phi \circ \psi)$-cellular automaton with memory set $\phi(S)T$. 
\end{proof}

\begin{lemma}\label{le-phi}
Let $\tau : A^G \to A^H$ be a $\phi$-equivariant function, for some $\phi \in \Hom(H,G)$.
\begin{enumerate}
 \item If $\tau$ is surjective, then $\phi$ is injective.
 \item If $\tau$ is injective, then $\phi$ is surjective.
 \end{enumerate}
\end{lemma}
\begin{proof}
\begin{enumerate}
\item Suppose that $\phi$ is not injective, so $N:= \ker(\phi) \neq \{ e_H \}$. By Lemma \ref{fix}, $\text{Im}(\tau) \subseteq \Fix(N)$. However, $\Fix(N) \subsetneq A^H$, as configurations such as $x \in A^H$ defined by $x(h) = 1$ if $h=e_H$ and $x(h)=0$ if $h \neq e_H$ are not in $\Fix(N)$. This shows that $\tau$ is not surjective. 

\item Suppose that $\phi$ is not surjective. The set of $G$-periodic configurations $\Fix(G)$ of $A^G$ corresponds to the set of constant configurations in $A^G$, so there are precisely $\vert A \vert$ of them. If $x \in \Fix(G)$, then, for every $h \in H$, 
\[ \tau(x)(h) = h^{-1} \cdot \tau(x)(e_H) = \tau(\phi(h^{-1}) \cdot x)(e_H) = \tau(x)(e_H). \]
This shows that $\tau(x)$ is constant, so $\tau(x) \in \Fix(H) \subseteq A^H$. 

The above proves that $\tau(\Fix(G)) \subseteq \Fix(H)$. If $\tau(\Fix(G)) \subsetneq \Fix(H)$, then $\vert \Fix(G) \vert = \vert A \vert = \vert \Fix(H) \vert$ implies that $\tau$ is not injective by the Pigeonhole Principle. Otherwise, suppose that $\tau (\Fix(G)) = \Fix(H)$. As $\phi$ is not surjective, $\phi(H)$ is a proper subgroup of $G$, and as $|A|\geq 2$, there exists a non-constant $z \in A^G$ that is $\phi(H)$-periodic (e.g. the indicator function $z: G \to A$ defined by $z(g)=1$ if $ g \in \phi(H)$ and $z(g) = 0$ if $g \not \in \phi(H)$). Again we have that for all $h \in H$,
\[ \tau(z)(h) = h^{-1} \cdot \tau(z)(e_H) = \tau(\phi(h^{-1}) \cdot z)(e_H) = \tau(z)(e_H). \]
Therefore, $\tau(z)$ is $H$-periodic, so $\tau(z) \in \Fix(H)$. This proves that $\tau$ is not injective, as $\tau(z)$ also has a preimage in $\Fix(G)$. 
\end{enumerate} 
\end{proof}

The converse of Lemma \ref{le-phi} is clearly not true; for example, when $G=H$, there are $\id_G$-cellular automata that are neither surjective nor injective.

We say that a $\phi$-cellular automaton $\tau : A^G \to A^H$ is \emph{invertible} if there exists a group homomorphism $\psi : G \to H$ and a $\psi$-cellular automaton $\sigma : A^H \to A^G$ such that $\tau \circ \sigma = \id_{A^H}$ and $\sigma \circ \tau = \id_{A^G}$. As usual, when such $\sigma$ exists, we write $\tau^{-1} := \sigma$.  

\begin{theorem}\label{invertible}
A $\phi$-cellular automaton $\tau : A^G \to A^H$ is invertible if and only if it is bijective. 
\end{theorem}
\begin{proof}
It is clear that if $\tau$ is invertible, then it must be bijective.

Suppose that $\tau$ is bijective and consider the inverse function $\tau^{-1} : A^H \to A^G$. As $\tau$ 
is continuous, $A^G$ is compact and $A^H$ is Hausdorff, it follows that $\tau^{-1}$ is continuous. Now, by Lemma \ref{le-phi}, $\phi : H \to G$ is an isomorphism. Let $y \in A^H$ and $g \in G$ be arbitrary. Then, there exists $x \in A^G$ and $h \in H$ such that $\tau(x) = y$ and $\phi(g) = h$. As $\tau$ is $\phi$-equivariant, we have
\[ h \cdot \tau(x) = \tau( \phi(h) \cdot x). \]
Applying $\tau^{-1}$ above and substituting $x=\tau^{-1}(y)$ and $h = \phi^{-1}(g)$, we obtain
\[ \tau^{-1} ( \phi^{-1}(g) \cdot y ) = g \cdot \tau^{-1}(y). \]
This proves that $\tau^{-1}$ is $\phi^{-1}$-equivariant. Therefore, $\tau^{-1}$ is a $\phi^{-1}$-cellular automaton by Theorem \ref{Curtis-Hedlund}. 
\end{proof}

\begin{corollary}
If $\tau :A^G \to A^H$ is an invertible $\phi$-cellular automaton, then $\tau^{-1} : A^H \to A^G$ is a $\phi^{-1}$-cellular automaton. 
\end{corollary}

\section{The group of generalized invertible cellular automata}

We begin this section by investigating further the generalized cellular automata defined in Example \ref{star}.

\begin{lemma}\label{L-1}
Let $G$, $H$ and $K$ be groups, and consider $\phi \in \Hom(H,G)$ and $\psi \in \Hom(K,H)$. Then, 
\[ (\psi \circ \phi)^\star = \phi^\star \circ \psi^\star\]
\end{lemma}
\begin{proof}
Observe that for any $x \in A^K$,
\[ (\psi \circ \phi)^\star (x) = x \circ ( \psi \circ \phi) = (x \circ \psi) \circ \phi = \phi^\star \circ \psi^\star (x).  \]
\end{proof}

We shall focus on studying $\phi$-cellular automata when $G=H$. Following the notation in \cite{CSC10}, let $\CA(G;A)$ be the set of all $\id_G$-cellular automata $\tau : A^G \to A^G$. This is a monoid equipped with the composition of functions. Let $\ICA(G;A)$ be the group of units of $\CA(G;A)$, i.e. the group consisting of all invertible $\id_G$-cellular automata over $A^G$. Define
\begin{align*}
\GCA(G;A) & := \{ \tau : A^G \to A^G  \; \vert \;  \tau \text{ is a $\phi$-cellular automaton for some } \phi \in \End(G) \},  \\
\IGCA(G;A) & := \{ \tau : A^G \to A^G \; \vert \; \tau \text{ is a bijective $\phi$-cellular automaton for some } \phi \in \Aut(G)  \}.  
\end{align*}
By Theorem \ref{composition}, $\GCA(G;A)$ is a monoid equipped with the composition of functions, and by Theorem \ref{invertible}, $\IGCA(G;A)$ is its group of units. 

Define $\End(G)^{\op}$ as the monoid with set $\End(G)$ and binary operation $\odot$ given by $\phi \odot \psi = \psi \circ \phi$, for every $\phi, \psi \in \End(G)$.  

\begin{lemma}
With the notation introduced above,
\begin{enumerate}
\item $\CA(G;A)$ is a submonoid of $\GCA(G;A)$
\item $\End(G)^{\op}$ is isomorphic to a submonoid of $\GCA(G;A)$. 
\end{enumerate}
\end{lemma}
\begin{proof}
The first part is clear by Theorem \ref{composition}. Now define a function $\Phi : \End(G)^{\op} \to \GCA(G;A)$ by $\Phi(\phi) = \phi^{\star}$ for every $\phi \in \End^{\op}(G)$. By the previous lemma, for any $\phi, \psi \in \End^{\op}(G)$, we have
\[ \Phi( \phi \odot \psi) = \Phi( \psi \circ \phi) = ( \psi \circ \phi)^\star = \phi^\star \circ \psi^\star = \Phi(\phi) \circ \Phi(\psi).  \]
This shows that $\Phi$ is a homomorphism of monoids. Finally, assume that $\Phi(\phi) = \Phi(\psi)$. Then $\phi^\star = \psi^\star$, which implies that for every $x \in A^G$, 
\[ x \circ \phi  = \phi^\star(x) = \psi^\star(x) = x \circ \psi.  \]
Thus, for every $g \in G$, and $x \in A^G$,
\[ x( \phi(g)) = x( \psi(g) ).  \]
For each $g \in G$, let $\chi_g : G \to A$ be the indicator function of $g$, i.e. $\chi_g(h) =1$ if $h=g$ and $\chi_g(h) = 0$ if $h \neq g$. Therefore, for any $g \in G$,
\[ 1 = \chi_{\phi(g)} (\phi(g)) = \chi_{\phi(g)}(\psi(g)).  \] 
This implies that $\phi(g) = \psi(g)$ for all $g \in G$, so $\phi = \psi$. 

\end{proof}

Abusing notation, we also denote by $\End(G)^{\op}$ the isomorphic copy inside $\GCA(G;A)$ of $\End(G)^{\op}$.

\begin{lemma}\label{normal}
$\ICA(G;A)$ is a normal subgroup of $\IGCA(G;A)$. 
\end{lemma}
\begin{proof}
Let $\sigma \in \ICA(G;A)$ and $\tau \in \IGCA(G;A)$. Then $\tau$ is a bijective $\phi$-cellular automaton for some $\phi \in \Aut(G)$ and $\tau^{-1}$ is a bijective $\phi^{-1}$-cellular automaton. By Theorem \ref{composition}, $\tau^{-1} \circ \sigma \circ \tau$ is a bijective $\phi^{-1} \circ \id_G \circ \phi = \id_G$-cellular automaton, so $\tau^{-1} \circ \sigma \circ \tau \in \ICA(G;A)$. 
\end{proof}

For subsets $A$ and $B$ of $\GCA(G;A)$, define $A \circ B := \{ \tau \circ \sigma : \tau \in A, \sigma \in B \}$. 

\begin{lemma}\label{product}
\text{\\}
\begin{enumerate}
\item $\GCA(G;A) =  \End(G)^{\op} \circ \CA(G;A)$.
\item $\IGCA(G;A) = \Aut(G)^{\op} \circ \ICA(G;A)$.
\end{enumerate}
\end{lemma}
\begin{proof}
\begin{enumerate}
\item Let $\tau \in \GCA(G;A)$. By definition, there exist $\phi \in \End(G)$, a finite subset $S \subseteq G$, and $\mu :A^S \to A$, such that $\tau(x)(g) = \mu( (\phi(g^{-1}) \cdot x) \vert_S)$, for all $g \in G$, $x \in A^G$. Define $\hat{\tau} \in \CA(G;A)$ by 
\[ \hat{\tau}(x)(g) := \mu(( g^{-1} \cdot x)\vert_S), \quad \forall g \in G, x \in A^G.  \]
Observe that, for all $x \in A^G$,
\[ (\phi^\star \circ \hat{\tau})(x) =\phi^\star( \hat{\tau}(x)) = \hat{\tau}(x) \circ \phi.    \]
Thus, for all $g \in G$, 
\[ (\phi^\star \circ \hat{\tau})(x)(g) = \hat{\tau}(x) ( \phi(g)) = \mu(( \phi(g^{-1}) \cdot x)\vert_S) =\tau(x)(g).  \]

\item Let $\tau \in \IGCA(G;A)$. By the previous point, there exist $\phi \in \End(G)$ and $\hat{\tau} \in \CA(G;A)$ such that $\tau = \phi^\star \circ \hat{\tau}$. By Lemma \ref{le-phi}, $\phi$ must be bijective, so $\phi \in \Aut(G)^{\op}$.  Furthermore, $\hat{\tau} = (\phi^{\star})^{-1} \circ \tau$, so $\hat{\tau}$ must be bijective. By Theorem \ref{invertible}, $\hat{\tau} \in \ICA(G;A)$.  
\end{enumerate}
\end{proof}

As noted in the previous section, a function $\tau : A^G \to A^H$ may be a $\phi$-cellular automaton and a $\psi$-cellular automaton for different homomorphisms $\phi, \psi \in \Hom(H,G)$. However, this may not happen whenever $\tau$ is injective, as shown in the following result. 

\begin{lemma}
Let $\tau : A^G \to A^H$ be an injective function that is $\phi$-equivariant and $\psi$-equivariant for some $\phi, \psi \in \Hom(H,G)$. Then, $\phi = \psi$. 
\end{lemma}
\begin{proof}
By Definition \ref{equiv}, for all $h \in H$, $x \in A^G$,
\[ \tau( \phi(h) \cdot x) = h \cdot \tau(x) = \tau( \psi(h) \cdot x).  \]
As $\tau$ is injective, then $\phi(h) \cdot x = \psi(h) \cdot x$ for all $h \in H$, $x \in A^G$. As $\vert A \vert \geq 2$, the shift action is faithful (see \cite[2.7.2]{CSC10}), so we have $\phi(h) = \psi(h)$ for all $h \in H$. The result follows. 
\end{proof}

Recall that a group $K$ is the \emph{semidirect product} of two subgroups $A$ and $B$ if $K = AB$, $A$ is normal in $K$ and $A \cap B = \{ e_K \}$  (see \cite[p. 167]{Rotman}). In such case, we write $K \cong A  \rtimes B$. 

\begin{theorem}
$\IGCA(G;A) \cong \ICA(G;A) \rtimes \Aut(G)^{op}$.
\end{theorem}
\begin{proof}
It follows by Lemma \ref{normal} that $ \ICA(G;A) $ is a normal subgroup of $\IGCA(G;A)$. Now Lemma \ref{product}, and the fact that $ \ICA(G;A) $ is normal, implies that
\[ \IGCA(G;A) = \Aut(G)^{op} \circ \ICA(G;A)  = \ICA(G;A) \circ \Aut(G)^{op}.  \]
Finally, let $\tau \in \Aut(G)^{op} \cap \ICA(G;A)$. Hence, there exist $\phi \in  \Aut(G)$ and $\hat{\tau} \in \ICA(G;A)$ such that $\tau = \phi^{\star} = \hat{\tau}$. This implies that $\tau$ is a $\phi$-equivariant and $\id_{G}$-equivariant. By the previous lemma, we must have that $\phi = \id_G$, so $\tau = (\id_G)^\star = \id_{A^G}$. The result follows. 
\end{proof}


\section{An application to the automorphisms of $\CA(G;A)$}

The key idea of this section is that every automorphism $\phi$ of $G$ induces an automorphism of $\CA(G;A)$ via conjugation by $\phi^\star$. Explicitly, for each $\phi \in \Aut(G)$, we define a map $\phi_{\CA} : \CA(G;A) \to \CA(G;A)$ by
\[  \phi_{\CA}(\tau)  : = (\phi^{-1})^\star  \circ \tau \circ \phi^\star, \quad \forall \tau \in \CA(G;A). \]

It follows by the same argument as in the proof of Lemma \ref{normal} that $\phi_{\CA}(\tau) \in \CA(G;A)$.

\begin{proposition}
Let $\phi, \psi \in\Aut(G)$. 
\begin{enumerate}
\item $(\phi \circ \psi)_{\CA} = \phi_{\CA} \circ \psi_{\CA}$. 
\item $\phi_{\CA} \in \Aut(\CA(G;A))$.  
\end{enumerate}
\end{proposition}
\begin{proof}
\begin{enumerate}
\item  Lemma \ref{L-1} implies that for any $\tau \in \CA(G;A)$,
\begin{align*}
(\phi \circ \psi)_{\CA}(\tau) & =  ((\phi \circ \psi)^{-1})^\star \circ \tau \circ (\phi \circ \psi)^\star \\
& = (\psi^{-1} \circ \phi^{-1})^\star \circ \tau \circ (\phi \circ \psi)^\star  \\
& =   (\phi^{-1})^\star\circ (\psi^{-1})^\star\circ \tau \circ \psi^\star \circ \phi^\star \\
& = \phi_{\CA} \circ \psi_{\CA}(\tau). 
\end{align*}

\item It is easy to check that $\phi_{\CA}$ is a homomorphism of $\CA(G;A)$ as it is induced by conjugation by $\phi^\star$. Moreover, it is an automorphism as, by part two, the inverse of $\phi_{\CA}$ is $(\phi^{-1})_{\CA}$. 
\end{enumerate}
\end{proof}

\begin{corollary}
The map $\Phi: \Aut(G)\rightarrow \Aut(\CA(G;A))$, defined as $\Phi(\phi) = \phi_{\CA}$ for every $\phi \in \Aut(G)$, is a group homomorphism.  
\end{corollary}

When $G = \mathbb{Z}$, and $\phi \in \Aut(\mathbb{Z})$ is the only nontrivial automorphism of $\mathbb{Z}$, i.e. $\phi(k)=-k$, $\forall k \in \mathbb{Z}$, then $\phi_{\CA}$ is the so-called \emph{mirrored rule}, or \emph{reflection}, which is widely used in the study of elementary cellular automata (see \cite[Ch. 20]{Wolfram}). It was already proved in \cite[Sec. 3]{CRG20} that the mirrored rule is indeed an automorphism of $\CA(G;A)$. 

Recall that an automorphism $\alpha$ of a monoid $M$ is \emph{inner} if there exist an invertible element $a \in M$ such that $\alpha(x) = a^{-1} x a$, for all $x \in M$. In other words, the inner automorphisms of $M$ are the ones induced by conjugation by the elements in the group of units $U(M)$. Let $\text{Inn}(M)$ be the set of inner automorphisms of $M$. It is well-known that $\text{Inn}(M)$ is a normal subgroup of $\text{Aut}(M)$ and that $\text{Inn}(M) \cong U(M) / Z(U(M))$, where $Z(U(M))$ is the center of the group $U(M)$.   

\begin{theorem}\label{inner}
Let $\phi \in \Aut(G)$. If $\phi_{CA}$ is an inner automorphism of $\CA(G;A)$, then $\phi(z) = z$, for all $z \in Z(G) := \{ z \in G : zg=gz, \forall g \in G \}$.
\end{theorem}
\begin{proof}
Suppose that $\phi_{\CA}$ is an inner automorphism. Then exist $\sigma\in \ICA(G;A)$ such that $\phi_{\CA}(\tau) = \sigma^{-1}\tau \sigma$ for all $\tau\in \CA(G;A)$. For each $z \in Z(G)$, define $\tau_z : A^G \to A^G$ by $\tau_z(x) := z \cdot x$, for all $x \in A^G$. It its clear that $\tau_z$ is continuous and $\id_G$-equivariant, so $\tau_z \in \CA(G;A)$. Moreover, it is easy to check that $\tau_z$ commutes with every element of $\CA(G;A)$. Hence, 
\[  (\phi^{-1})^\star  \circ \tau_z \circ \phi^\star = \phi_{\CA}(\tau_z) = \sigma^{-1}\tau_z \sigma = \sigma^{-1}\sigma \tau_z = \tau_z.  \]
Then, for all $x \in A^G$,
\[ \tau_z \circ \phi^\star(x)  =  \phi^\star \circ \tau_z(x) \ \Rightarrow \ z \cdot \phi^\star(x) = \phi^\star(z \cdot x).  \]
By Lemma \ref{L-1},  
\[ \phi^\star(\phi(z) \cdot x) =  z \cdot \phi^\star(x) = \phi^\star(z \cdot x). \] 
As $\phi \in \Aut(G)$, then $\phi^\star$ is bijective, so $\phi(z) \cdot x = z \cdot x$ for all $x \in A^G$. Since $\vert A \vert \geq 2$, the action of $G$ on $A^G$ is faithful, so $\phi(z) = z$ for all $z \in Z(G)$. 
\end{proof}

\begin{corollary}
The mirrored rule is not an inner automorphism of $\CA(\mathbb{Z};A)$. \end{corollary}

Define the \emph{outer automorphism group} of a monoid $M$ by $\text{Out}(M) :=\Aut(M) / \text{Inn}(M)$.

\begin{corollary}
Suppose that $G$ is an abelian group. Then, the homomorphism $\Psi : \Aut(G) \rightarrow \text{Out}(\CA(G;A))$ given by $\Psi(\phi) = \phi_{\CA} \text{Inn}(\CA(G;A))$ is injective.  
 \end{corollary}
\begin{proof}
Suppose that $ \phi_{\CA} \text{Inn}(\CA(G;A)) = \psi_{\CA} \text{Inn}(\CA(G;A))$ for some $\psi, \phi \in \Aut(G)$. Then $\psi_{\CA}^{-1} \circ \phi_{\CA}$ is an inner automorphism of $\CA(G;A)$. By Theorem \ref{inner}, $\psi^{-1} \circ \phi (z) = z$ for all $z \in Z(G)=G$, which shows that $\psi = \phi$.  
\end{proof}

The previous result shows that, when $G$ is abelian, $\text{Out}(\CA(G;A))$ contains a subgroup isomorphic to $\Aut(G)$. Inspired by this, we propose the following question:

\begin{question}
For a given abelian group $G$, is it possible to show that $\text{Out}(\CA(G;A))$ is isomorphic to $\Aut(G)$, or can we find an element of $\text{Out}(\CA(G;A))$ that is not induced by any automorphism of $G$? 
\end{question}


\section*{Acknowledgments}

The first author was supported by a CONACYT Basic Science Grant (No. A1-S-8013). The second and third authors were supported by CONACYT National Posgraduate Scholarships.


\end{document}